\documentclass[12pt]{amsart}
\usepackage{amsmath}
\usepackage{amssymb}
\usepackage{amsthm}
\usepackage{amscd}
\usepackage[most]{tcolorbox}
\usepackage{comment}
\usepackage{tikz-cd} 
\usepackage{tikz}
\usetikzlibrary{arrows.meta,calc,chains,positioning,quotes,shapes,shapes.multipart}
\usetikzlibrary{shapes,arrows}
\usepackage{url}

\tikzstyle{myarrow}=[->, >=open triangle 90, thick]
\tikzstyle{line} = [draw, -latex']
\tikzstyle{Vertex} = [draw, circle, minimum width = .25cm, text width=.25cm,  text centered]
\tikzstyle{box} = [draw, rectangle, minimum width = 1.5cm, text width=.25cm,  text centered]

\newtheorem{theorem}{Theorem}[section]

\newtheorem{lemma}[theorem]{Lemma}
\newtheorem{corollary}[theorem]{Corollary}

\theoremstyle{definition}
\newtheorem{definition}[theorem]{Definition}

\allowdisplaybreaks

\theoremstyle{remark}
\newtheorem{remark}[theorem]{Remark}

\numberwithin{equation}{section}

\begin{document}
	
	
	\title[CP Maps on Hilbert modules over Pro C*-algebras]{Completely Positive Maps: Pro-$C^*$-algebras and Hilbert Modules over Pro-$C^*$-algebras}


	\author{Bhumi Amin}
	\address{Department of Mathematics, IIT Hyderabad, Telangana, India - 502285 }
	\email{ma20resch11008@iith.ac.in}
	
	\author{Ramesh Golla}
	\address{Department of Mathematics, IIT Hyderabad, Telangana, India - 502285}
	
	\email{rameshg@math.iith.ac.in}

	
	\subjclass[2020]{Primary: 46L05,\;46L08;\; Secondary: 46K10}
	
	
	
	\keywords{completely positive maps, pro-$C^*$-algebra, Hilbert modules, Stinesping's dilation}
	
	\begin{abstract}
		In this paper, we describe the structure of completely positive maps between two pro-$C^*$-algebras using Paschke's GNS construction for CP-maps on pro-$C^*$-algebras. Furthermore, we establish a structure theorem for a $\phi$-map between two Hilbert modules over pro-$C^*$-algebras, where $\phi$ is a continuous CP-map between pro-$C^*$-algebras. We also discuss the minimality of these representations.
	\end{abstract}
	
	\maketitle

		\section{Introduction}
	Completely positive maps have numerous applications in modern mathematics, including quantum information theory, statistical physics, and stochastic processes. Stinespring's representation theorem stands as a cornerstone in the study of completely positive maps, a vital concept within the theory of $C^*$-algebras. It is a structure theorem that shows how to represent completely positive maps from a $C^*$-algebra into the $C^*$-algebra of bounded operators on a Hilbert space. 
	\begin{theorem}\cite{WS,VPaulsen}
		Let $\mathcal{A}$ be a unital $C^*$-algebra and $H$ be a Hilbert space. Let $\phi: \mathcal{A} \rightarrow B(H)$ be a completely positive map. Then there exists a Hilbert space $K$, a bounded linear operator $V:H\rightarrow K$, and a $^*-$representation $\pi: A \rightarrow B(K)$ such that 
		$$\phi(a) = V^*\pi(a)V,$$
		for all $a \in \mathcal{A}.$
	\end{theorem}
	One can view Stinespring's representation theorem as a natural extension of the renowned Gelfand-Naimark-Segal theorem for states on $C^*$-algebras (see \cite[pg. 278]{KR}). In 1972, Paschke (in \cite{Pash1}) initially gave the structure theorem for completely positive maps from a $U^*$-algebra to the $C^*$-algebra of bounded operators. Later, he generalized his work to completely positive maps from a $U^*$-algebra to a $C^*$-algebra (see \cite{Pash2} for more details).

	Hilbert modules over $C^*$-algebras extend the concept of Hilbert space by allowing the inner product to have values in a $C^*$-algebra. The idea of a Hilbert module over a unital, commutative $C^*$-algebra was initially introduced by Kaplansky in \cite{KI}. Asadi, in \cite{AMB}, provided a Stinespring-like representation for operator-valued completely positive maps on Hilbert $C^*$-modules.  In 2012, Bhat, Ramesh, and Sumesh enhanced Asadi's result by eliminating a technical condition (see \cite{BRS}). Using induced representations of Hilbert $C^*$-modules, Skeide (in \cite{MS2}) obtained a factorization theorem generalizing \cite[Theorem 2.1]{BRS}.
	
	In 1971, A. Inoue introduced locally $C^*$-algebras to generalize the concept of $C^*$-algebras (see \cite{AI} for more details). A complete topological involutive algebra with a topology induced by a family of semi-norms is called a locally $C^*$-algebra. It has been explored with alternative names, one such being ``pro-$C^*$-algebra", a term we will use throughout the paper. Later,  in 1988, Phillips considered them as projective limits within inverse families of $C^*$-algebras (see \cite{NCP}).
	
	The theory of completely positive maps on pro-$C^*$-algebras has been extensively investigated and documented in the book \cite{MJbook}, and in the paper \cite{MJ}, both authored by Joita. In 2012, Maliev and Pliev extended the techniques from \cite{BRS}, originally applied to $C^*$-algebras, to establish a Stinespring theorem for Hilbert modules over pro-$C^*$-algebras using the theory of local Hilbert spaces. In 2017, Karimi and Sharifi generalized Paschke's GNS construction to completely positive maps on pro-$C^*$-algebras. Using this generalization, they derived an analogue of Stinespring theorem for Hilbert modules over pro-$C^*$-algebras (refer to \cite{KS} for further details). 
	
	In this paper, we begin by introducing key definitions and fundamental concepts regarding pro-$C^*$-algebras and Hilbert modules over them. We employ continuous functional calculus for pro-$C^*$-algebras in certain computations (see \cite[Proposition 1.1.14]{MJbook} for more details). With the help of this, along with Paschke's GNS construction for pro-$C^*$-algebras, we derive a variant of Stinespring's theorem. Lastly, using a generalization of \cite[Lemma 1]{KS}, we obtain a Stinespring-like theorem for Hilbert modules over pro-$C^*$-algebras. Note that the differentiating factor here is that the characterization is given for maps between two pro-$C^*$-algebras and for maps between two Hilbert modules over respective pro-$C^*$-algebras using dilated spaces.

	\section{Preliminaries}

	Throughout the paper, we consider unital algebras over complex field. Let us first look into the definitions of pro-$C^*$-algebras and Hilbert modules over them.
	
	\begin{definition}\cite{AI}
		A $^*-$algebra $\mathcal{A}$ is called a pro-$C^*$-algebra if there exists a family $\{p_j\}_{j\in J}$ of semi-norms defined on $\mathcal{A}$ such that: 
		\begin{enumerate}
			\item $\{p_j\}_{j\in J}$ defines a complete Hausdorff locally convex topology on $\mathcal{A}$.
			\item $p_j(x y) \leq p_j(x) p_j(y),$ for all $x, y \in \mathcal{A}$ and each $j \in J$.
			\item $p_j(x^*x) = p_j(x)^2,$  for all $x \in \mathcal{A}$  and each $j \in J$.
		\end{enumerate}
	\end{definition}
	We call the family $\{p_j\}_{j\in J}$ of semi-norms defined on $\mathcal{A}$ as the family of $C^*$-semi-norms.
	
	Let $S(\mathcal{A})$ denote the set of all continuous $C^*$-semi-norms on $\mathcal{A}.$
	\begin{definition}\cite{NCP}
		Let $\mathcal{A}$ be a pro-$C^*$-algebra and $E$ be a complex vector space which is also a right $\mathcal{A}$-module. Then $E$ is said to be a pre-Hilbert $\mathcal{A}$-module if it is equipped with an $\mathcal{A}-$valued inner product $\langle .,. \rangle : E \times E \rightarrow \mathcal{A}$ which is $\mathbb{C}-$linear and $\mathcal{A}-$linear in the second variable and it satisfies the following conditions:
		\begin{enumerate}
			\item $\langle \xi, \eta \rangle^* = \langle \eta, \xi \rangle,$ for all $\xi, \eta \in E.$
			\item $\langle \xi, \xi \rangle \geq 0,$ for all $\xi \in E.$
			\item For any $\xi \in E,$ $\langle \xi, \xi \rangle = 0$ if and only if $\xi = 0.$
		\end{enumerate}
		We say that $E$ is a Hilbert $\mathcal{A}$-module if $E$ is complete with respect to the topology induced by the $C^*$-semi-norms $\{\|.\|_p\}_{p \in S(\mathcal{A})},$ 
		where, for any $\xi \in E$, 
		$$\|\xi\|_p := \sqrt{p(\langle \xi, \xi \rangle)}.$$
	\end{definition}
	We use the notation $\|.\|_{p_E}$ instead of $\|.\|_p$ when working with multiple Hilbert modules over the same pro-$C^*$-algebra.
	
	Next, let us briefly revisit the definition of a completely positive map on pro-$C^*$-algebras.
	\begin{definition}
		Let $\mathcal{A}$ and $\mathcal{B}$ be two pro-$C^*$-algebras. A linear map $\phi: \mathcal{A} \rightarrow \mathcal{B}$ is said to be completely positive (or CP), if for all $n \in \mathbb{N}, \phi^{(n)}: M_n(\mathcal{A}) \rightarrow M_n(\mathcal{B})$ defined by $$\phi^{(n)}([a_{ij}]_{i,j=1}^n) = [\phi(a_{ij})]_{i,j=1}^n$$ is positive.
	\end{definition}
	
	\begin{definition}\cite{MJbook}
		Let $\mathcal{A}$ and $\mathcal{B}$ be two pro-$C^*$-algebras. A $^*$-morphism from $\mathcal{A}$ to $\mathcal{B}$ is a linear map $\phi: \mathcal{A} \rightarrow \mathcal{B}$ such that:
		\begin{itemize}
			\item[(1)] $\phi(ab) = \phi(a)\phi(b),$ for all $a, b \in \mathcal{A}$ 
			\item[(2)] $\phi(a^*) = \phi(a)^*,$ for all $a \in \mathcal{A}.$
		\end{itemize}
	\end{definition}
	
	Later, we delve deeper into the following version of a renowned definition, incorporating modifications as necessary.
	\begin{definition}
		Let $\mathcal{A}$ and $\mathcal{B}$ be pro-$C^*$-algebras. Let $E$ be a Hilbert $\mathcal{A}$-module and $F$ be a Hilbert $\mathcal{B}$-module.

		Let $\phi: \mathcal{A} \rightarrow \mathcal{B}$ be a linear map. A map $\Phi: E \rightarrow F$ is said to be 
		\begin{enumerate}
			\item a $\phi$-map if $$\langle \Phi(x), \Phi(y) \rangle = \phi(\langle x , y \rangle), $$ for all $x, y \in E$.
			\item a $\phi$-morphism if $\Phi$ is a $\phi$-map and $\phi$ is a $^*$-morphism.
			\item  completely positive if $\Phi$ is a $\phi$-map and $\phi$ is completely positive.
		\end{enumerate}
	\end{definition}

	The set $\langle E, E \rangle$ is the closure of the linear span of $\{\langle x,y\rangle: x,y \in E\}.$ If $\langle E, E \rangle = \mathcal{A}$ then $E $ is said to be full.

	Let $E$ and $F$ be Hilbert modules over a pro-$C^*$-algebra $\mathcal{B}.$ We say that a map $T: E \rightarrow F$ is adjointable if there exists a map $T^*: F \rightarrow E$ satisfying \begin{equation*}
		\langle T\xi, \eta \rangle = \langle \xi, T^*\eta \rangle,
	\end{equation*}
	for all $\xi\in E$ and $\eta\in F$.
	Let $\mathcal{L}_\mathcal{B}(E,F)$ denote the set consisting of all continuous adjointable $\mathcal{B}$-module operators from $E$ to $F$. We denote $\mathcal{L}_\mathcal{B}(E,E)$ by $\mathcal{L}_\mathcal{B}(E).$  We can define inner-product on $\mathcal{L_B}(E,F)$ by 
	\begin{equation*}
		\langle T, S \rangle := T^*S,\; \text{ for}\; T, S \in \mathcal{L}_\mathcal{B}(E,F).
	\end{equation*}
	
	Note that $\mathcal{L}_\mathcal{B}(E,F)$ is a Hilbert $\mathcal{L}_\mathcal{B}(E)$-module with the module action 
	
	\begin{equation*}
		(T,S) \rightarrow TS, \; \text{ for} \; T \in \mathcal{L}_\mathcal{B}(E,F)\; \text{ and} \;S \in {L}_\mathcal{B}(E).
	\end{equation*}

    \begin{remark}
        An operator $T \in \mathcal{L}_\mathcal{B}(E,F)$ is said to be an isometry if $T^*T = I_E.$ Here $I_E$ denotes the identity operator on $E$,
    \end{remark}

	\begin{lemma}\cite[Lemma 2.1.3]{MJbook}
		Let $\mathcal{B}$ be a pro-$C^*$-algebra, and $E$ and $F$ be Hilbert $\mathcal{B}$-modules. Let $p \in S(\mathcal{B}).$ The map $\|.\|_{\tilde{p}} : \mathcal{L}_\mathcal{B}(E,F) \rightarrow [0,\infty)$ defined by 
		$$ \|T\|_{\tilde{p}} = sup\{\|Tx\|_{p_F}: \|x\|_{p_E}\leq 1, x \in E\}$$
		is a seminorm on $\mathcal{L}_\mathcal{B}(E,F).$ Moreover, $\tilde{p}$ is submultiplicative if $E=F.$
	\end{lemma}
	
	\begin{remark}\label{K}
		In fact, if $T \in \mathcal{L_B}(E,F)$ then there exists a real constant $K \geq 0$ such that
		$\langle Tx, x \rangle \leq K \langle x,x \rangle$
		for all $x \in E$ (see \cite[Proposition 2.2.7]{MJbook} and \cite[Theorem 2.8]{Pash2}).
	\end{remark}

	\begin{definition}\cite{MJbook}
		Let $\mathcal{A}$ and $\mathcal{B}$ be  pro-$C^*$-algebras. A Hilbert $\mathcal{B}$-module $E$ is called a Hilbert $\mathcal{A}\mathcal{B}$-module if there is a non-degenerate $^*$-morphism $\tau:\mathcal{A} \rightarrow \mathcal{L}_\mathcal{B}(E)$.
		
		Hence, we identify $a.e $ with $\tau(a).e$ for all $a \in \mathcal{A}$ and $e \in E.$
	\end{definition}
	By a Hilbert $\mathcal{B}$-module we mean Hilbert (right) $\mathcal{B}$-module, for any pro-$C^*$-algebra $\mathcal{B}.$

	\section{Stinespring representation theorem}
	

	Let $\mathcal{A}$ and $\mathcal{B}$ be  pro-$C^*$-algebras. 
	\begin{definition}\label{cpPhi}
		Let $K_1$ and $K_2$ be Hilbert $\mathcal{B}$-modules.
		Let $\phi: \mathcal{A} \rightarrow \mathcal{L_B}(K_1)$ be a linear map.
		A map $\Phi: E \rightarrow \mathcal{L}_\mathcal{B}(K_1,K_2)$ is said to be 
		
		\begin{enumerate}
			\item a $\phi$-map if  $$\langle \Phi(x), \Phi(y) \rangle =  {\Phi(x)}^*\Phi(y)= \phi(\langle x , y \rangle), $$ for all $x, y \in E$.
			\item a $\phi$-morphism if $\Phi$ is a $\phi$-map and $\phi$ is a $^*$-morphism.
			\item completely positive if $\Phi$ is a $\phi$-map and $\phi$ is completely positive.
		\end{enumerate}
		
	\end{definition}
	
	Note that the map $\Phi: E \rightarrow \mathcal{L}_\mathcal{B}(K_1,K_2)$ is said to be non-degenerate if the span closure (with respect to the right action by $\mathcal{B}$) of the sets $\Phi(E)K_1$ and $\Phi(E)^*K_2$ are $K_1$ and $K_2$ respectively, that is, 
	\begin{equation*}
		[\Phi(E)K_1]_\mathcal{B} = K_2 \text{ and }  [\Phi(E)^*K_2]_\mathcal{B} = F_1.
	\end{equation*}

	\begin{definition}\cite{MJbook}
		A representation of $\mathcal{A}$ on a Hilbert $\mathcal{B}$-module $F$ is a continuous $^*$-morphism $\Gamma$ from $\mathcal{A}$ to $\mathcal{L_B}(F)$. The representation is said to be non-degenerate if $\Gamma(\mathcal{A})F$ is dense in $F.$
	\end{definition}

	
	We now state Paschke's GNS construction for completely positive maps on  pro-$C^*$-algebras, which was given by Karimi and Sharifi in \cite{KS}. 
	
	\begin{theorem}\cite[Theorem 1]{KS}\label{Paschpro} 
		Let $\mathcal{A}, \mathcal{B}$ be  pro-$C^*$-algebras and $\phi:\mathcal{A} \rightarrow \mathcal{B}$ be a continuous completely positive map. Then there exists a Hilbert $\mathcal{B}$-module $X,$ a unital continuous representation $\pi_\phi: \mathcal{A} \rightarrow \mathcal{L}_\mathcal{B}(X),$ and an element $\xi \in X$ such that 
		$$\phi(a) = \langle \xi, \pi_\phi(a) \xi \rangle,$$
		for all $a \in \mathcal{A}$. Moreover, the set $\chi_\phi = \text{span}\{\pi_\phi(a)(\xi b): a \in \mathcal{A}, b \in \mathcal{B}\}$ is a dense subspace of $X.$
	\end{theorem}

	\section{Representation of CP-maps between Pro-$C^*$-algebras}
	
	The following theorem provides an adaptation of the Stinespring representation theorem for pro-$C^*$-algebras. 
	
	\begin{theorem}\label{prorep1}
		Let $\mathcal{A}$, $\mathcal{B}$ be  pro-$C^*$-algebras. Let $\phi: \mathcal{A} \rightarrow \mathcal{B} $ be a continuous, unital, completely positive map. Then there exists 
		\begin{enumerate}
			\item Hilbert $\mathcal{B}$-modules $D$ and $K_\phi$, 
			\item a unital representation $\pi_\phi: \mathcal{A} \rightarrow \mathcal{L}_\mathcal{B}(K_\phi),$
			\item a bounded linear operator $V_\phi \in \mathcal{L}_\mathcal{B}(D,K_\phi)$
		\end{enumerate}
		such that 
		\begin{equation*}
			\phi(a)I_D = {V_\phi}^*\pi_\phi(a)V_\phi, \; \text{for all}\; a \in \mathcal{A}.
		\end{equation*}Here $I_D$ is the identity mapping on $D$. Moreover, $[\pi_\phi(\mathcal{A})V_\phi (D)]_\mathcal{B} = K_\phi.$
	\end{theorem}
	
	\begin{proof}
		Let $\tilde{\pi}_\phi$ be the Paschke's GNS construction associated with $\phi$ as in Theorem \ref{Paschpro}. That is, we get a Hilbert $\mathcal{B}$-module $X$ with a $\mathcal{B}$-valued inner product $\langle .,. \rangle_X$, an element $\xi \in X$ and a representation $\tilde{\pi}_\phi : \mathcal{A} \rightarrow \mathcal{L}_\mathcal{B}(X)$ such that $\phi(a) = \langle \xi, \tilde{\pi}_\phi(a)\xi \rangle_X$, for all $a \in \mathcal{A}$.
		
		Define $\mathcal{B}$-modules:  
\[		{D_1}' := X \otimes \mathcal{B} = [X \otimes 1_\mathcal{B}]_\mathcal{B}, \quad {D_2}' := [\xi \otimes 1_\mathcal{B}]_\mathcal{B}, \]
		with the right action given by,
		$$ (x \otimes b)\beta := x \otimes b\beta,$$
		for all $\beta, b \in 
		\mathcal{B}, x \in X.$
		Define the standard inner product on ${D_i}',$ $i= 1,2,$
		\[ 
		\langle x_1 \otimes b_1, x_2 \otimes b_2 \rangle = {b_1}^* \langle x_1, x_2 \rangle b_2, 
		\]
		for $b_1, b_2 \in \mathcal{B}$ and $x_1, x_2 \in X.$
		Then ${D_1}'$ and ${D_2}'$ are semi-Hilbert $\mathcal{B}$-modules. 
		
		Let $D_1, D_2$ denote the Hilbert $\mathcal{B}$-modules associated with ${D_1}',{D_2}'$ respectively. That is, for $i=\{1,2\}$, $D_i = \overline{D_i/N^i_{\langle .,.\rangle}}^{\langle .,.\rangle},$ where $$N^1_{\langle .,.\rangle} = \{x \otimes b \in {D_1}': \langle x \otimes b, x \otimes b \rangle = 0\}$$ and
		$$N^2_{\langle .,.\rangle} = \{\xi \otimes b \in {D_2}': \langle \xi \otimes b, \xi \otimes b \rangle = 0\}.$$
		
		
		Moreover, define a map $\tau: \mathcal{B} \rightarrow \mathcal{L}_\mathcal{B}(D_2)$ by $\beta \rightarrow L_\beta :=\tau(\beta),$ where
$$    L_\beta(\xi \otimes b) := \xi \otimes \beta b ,$$
		for all $b \in \mathcal{B}$. Then $\tau$ is a continuous $*$-morphism.
		

		Indeed, for $b_1,b_2, \beta \in \mathcal{B}$, we have
		\[
		\begin{split}
			\langle L_\beta(\xi \otimes b_1), \xi \otimes b_2 \rangle 
			&= \langle \xi \otimes \beta b_1, \xi \otimes b_2 \rangle \\
			&= (\beta b_1)^* \langle \xi, \xi \rangle b_2 \\
			&= b_1^* \beta^* b_2 \\
			&= \langle \xi \otimes b_1,  \xi \otimes \beta^* b_2 \rangle \\ 
			&= \langle \xi \otimes b_1, \beta^* \cdot (\xi \otimes b_2) \rangle.
		\end{split}
		\]
		Hence $L_{\beta}$ is adjointable with adjoint $L_{\beta}^* = L_{\beta^*}$. 
		
		Thus $D_2$ has a well-defined adjointable left $\mathcal{B}$-action, and hence we say that $D_2$ is the Hilbert $\mathcal{BB}$-module associated with ${}_\mathcal{B}[\xi \otimes 1_\mathcal{B}]_\mathcal{B}$.

	Consider $D_1 \otimes_\tau D_2$	with the inner product $\langle \cdot, \cdot \rangle^\prime : D_1 \otimes_\tau D_2 \times D_1 \otimes_\tau D_2 \rightarrow \mathcal{B}$ defined by 
	\begin{equation}\label{ipdef}
	\langle d_1 \otimes d_2, {d_1}' \otimes {d_2}' \rangle^\prime :=  \langle d_2, \tau (\langle d_1, {d_1}' \rangle) {d_2}' \rangle =  \langle d_2, \langle d_1, {d_1}' \rangle {d_2}' \rangle,
	\end{equation}
	for $d_1, {d_1}' \in D_1$ and $d_2, {d_2}' \in D_2.$

	 Let $K_\phi$ denote the Hilbert $\mathcal{B}$-module associated with $D_1 \otimes_\tau D_2.$ That is, 
     $$  K_\phi = \overline{(D_1 \otimes_\tau D_2)/N_{\langle .,.\rangle^\prime}}^{\langle .,.\rangle^\prime}, $$ 
     where $$N_{\langle .,.\rangle^\prime} = \{d_1 \otimes d_2 \in D_1 \otimes D_2: \langle d_1 \otimes d_2, d_1 \otimes d_2 \rangle' = 0\}.$$
		
		
		Put 
$D := D_2.$ 
		
		For each $d \in D_1$, define $M_d:D \rightarrow K_{\phi}$ by 
		\begin{equation*}
			M_d(d_0) = d \otimes d_0,\; \text{for all}\; d_0 \in D.
		\end{equation*}
		Clearly, $M_d$ is $\mathcal{B}$-linear, in fact $M_d \in \mathcal{L}_\mathcal{B}(D,K_\phi).$ Indeed, for  $q \in S(\mathcal{B}),$ we have
		\begin{equation}
			\begin{split}
				\|M_d(d_0)\|_q^2 
				&= q(\langle d \otimes d_0, d \otimes d_0 \rangle') \\ 
				&= q(\langle d_0, \tau(\langle d,d \rangle) d_0 \rangle) \\
				& \leq K q(\langle d_0, d_0 \rangle)   \quad \quad (\text{by  Remark }\ref{K}).
			\end{split}
		\end{equation}
		Also, for $d_1 \in D_1,\text{ and } d_2, d_0 \in D,$ we have
		\begin{equation*}
			\begin{split}
				\langle  d_1 \otimes d_2 , M_d (d_0) \rangle' &= \langle d_1 \otimes d_2 , d \otimes d_0 \rangle'\\
				&= \langle  d_2, \tau(\langle d_1 , d \rangle) d_0 \rangle \\
				&= \langle  \tau(\langle d , d_1 \rangle) d_2,  d_0 \rangle \\
					&= \langle  \langle d , d_1 \rangle d_2,  d_0 \rangle \qquad \ \ (\text{by equation } (\ref{ipdef})) \\
				&= \langle {M_d}^*(d_1 \otimes d_2), d_0  \rangle.
			\end{split}
		\end{equation*}
		Hence, $M_d \in \mathcal{L_B}(D, K_\phi)$ with ${M_d}^*(d_1 \otimes d_2)=\langle d , d_1 \rangle d_2$ for all $d_1 \in D_1$ and $d_2\in D$.

		Now, we can define a map $\eta: D \rightarrow \mathcal{L}_\mathcal{B}(D, K_\phi)$ by $d \mapsto \eta(d) := M_d.$
		Take $V_\phi := \eta(\xi \otimes 1_\mathcal{B})$.  Then
		the adjoint of $V_\phi,$ is given by 
		$$ V_\phi^*(d_1 \otimes d_2) = \eta(\xi \otimes 1_\mathcal{B})^*(d_1 \otimes d_2) = {M_{\xi \otimes 1_\mathcal{B}}}^*(d_1 \otimes d_2) = \langle \xi \otimes 1_\mathcal{B}, d_1\rangle d_2,$$
		for all $d_1 \in D_1$ and $d_2 \in D.$

			Let $T \in \mathcal{L}_\mathcal{B}(D_1),$ define $\rho^\prime(T)$ on $D_1 \otimes_\tau D_2$ by 
		$$\rho^\prime(T)(d_1 \otimes d_2) := Td_1 \otimes d_2,$$ for all $d_1 \in D_1, d_2 \in D.$
		Observe that $\rho^\prime(T)(d_1 \otimes d_2) \in N_{\langle \cdot,\cdot \rangle'}$ whenever $d_1 \otimes d_2 \in N_{\langle \cdot,\cdot \rangle'}.$ Indeed, if $\langle d_1 \otimes d_2, d_1 \otimes d_2 \rangle' = 0$, then $\langle d_2,\tau(\langle d_1,d_1\rangle)d_2 \rangle = 0. $ Observe that, for any $p \in S(\mathcal{A})$, we have
		\begin{equation*}
			\begin{split}
				\langle \rho^\prime(T)(d_1 \otimes d_2), \rho^\prime(T)(d_1 \otimes d_2) \rangle' 
				&= \langle d_2,\tau(\langle Td_1,Td_1\rangle)d_2 \rangle  \\
			&= \langle d_2,\tau(\langle d_1,T^*Td_1\rangle)d_2 \rangle  \\
				&\leq K \langle  d_2,\tau(\langle d_1,d_1\rangle)d_2 \rangle \quad  \quad (\text{by Remark }\ref{K}) \\
				&=0.
			\end{split}
		\end{equation*}
		Therefore, this map induces a map $\tilde{\rho}(T)$ on $(D_1 \otimes_\tau D_2)/N_{\langle \cdot, \cdot \rangle}.$ Next, we observe that $\tilde{\rho}(T)$ is bounded using the above inequalities.	For the sake of completeness, for $d_1 \in D_1, d_2 \in D$ and $q \in S(\mathcal{B}),$ we have
		\begin{equation*}
			\begin{split}
				\|\tilde{\rho}(T)(d_1 \otimes d_2)\|_q^2 &=	q(\langle \tilde{\rho}(T)(d_1 \otimes d_2), \tilde{\rho}(T)(d_1 \otimes d_2) \rangle') \\
				&= q(\langle d_2, \tau(\langle Td_1,Td_1\rangle)d_2 \rangle)\\
				&\leq K\|d_1 \otimes d_2\|_q^2.
			\end{split}
		\end{equation*}	
		Hence $\tilde{\rho}(T)$ can be extended to a bounded linear operator $\rho_\phi(T)$ on $K_\phi$. 
		
		Note that, for $d_1, {d_1}' \in D_1, d_2, {d_2}' \in D$, we have
		\begin{align*}
			\langle d_1 \otimes d_2, \rho_\phi(T)({d_1}' \otimes {d_2}') \rangle' &=	\langle d_1 \otimes d_2, T{d_1}' \otimes {d_2}' \rangle' \\
			&= \langle d_2, \tau(\langle d_1, T{d_1}' \rangle) {d_2}' \rangle \\  
			&= \langle d_2, \langle T^* d_1, {d_1}' \rangle {d_2}' \rangle \\   
			&=	\langle \rho_\phi(T^*)(d_1 \otimes d_2), {d_1}' \otimes {d_2}' \rangle'.     	
	                               \end{align*}
		Hence, $\rho_\phi(T)^* = \rho_\phi(T^*)$.
		Therefore $\rho_\phi: \mathcal{L}_\mathcal{B}(D_1) \rightarrow \mathcal{L_B}(K_\phi)$	defined by 
		\begin{equation*}
			T \rightarrow \rho_\phi(T)
		\end{equation*}
		is a continuous $^*$-morphism and hence is a representation of $\mathcal{L}_\mathcal{B}(D_1)$ on $K_\phi$.

		Next, for $S \in \mathcal{L}_\mathcal{B}(X)$, define 
		$$ \tilde{\gamma_\phi}(S)(\xi \otimes b) = S\xi \otimes b, $$
		for all $b \in \mathcal{B}$. Clearly $\tilde{\gamma_\phi}(S)$ is $\mathcal{B}$-linear. Moreover, the map is well-defined as $\langle \xi \otimes b, \xi \otimes b \rangle = 0$ implies $\langle \tilde{\gamma_\phi}(S)(\xi \otimes b), \tilde{\gamma_\phi}(S)(\xi \otimes b) \rangle = 0$, since 
		\begin{equation*}
			\begin{split}
				\langle \tilde{\gamma_\phi}(S)(\xi \otimes b), \tilde{\gamma_\phi}(S)(\xi \otimes b) \rangle &= \langle S\xi \otimes b, S\xi \otimes b \rangle \\
				&= b^*\langle S\xi, S\xi \rangle b \\
				&\leq K b^* \langle \xi, \xi \rangle b \quad \quad (\text{by Remark }\ref{K})
				\\
				&\leq K \langle \xi \otimes b, \xi \otimes b \rangle \\
				&= 0.
			\end{split}
		\end{equation*}
		By the above calculations, we can in fact observe that, for each $S \in \mathcal{L}_\mathcal{B}(X),$ the map $\tilde{\gamma_\phi}(S)$ is continuous. Moreover, $\tilde{\gamma_\phi}(S)^* =\tilde{\gamma_\phi}(S^*).$ Thus, we can continuously extend $\tilde{\gamma_\phi}(S)$ to $\mathcal{L}_\mathcal{B}(D_1),$ say $\gamma_\phi(S) \in \mathcal{L}_\mathcal{B}(D_1)$, for each $S \in \mathcal{L}_\mathcal{B}(X).$

		Define $\pi_\phi:=\rho_\phi\circ\gamma_\phi\circ\tilde{\pi}_\phi.$ Since $\rho_\phi, \gamma_\phi$ and $\tilde{\pi}_\phi$ are continuous $*$-morphisms, their composition $\pi_\phi$ is also a continuous $*$-morphism. The map can be well understood with the help of the following diagram:

		\tikzstyle{line} = [draw, -latex'] 
		
		\begin{figure}[hbt!]
			\centering
			\begin{tikzpicture}[scale=0.60, arr/.style = {-Stealth}]
				\node (3) at (8,0) {$\mathcal{A}$};
				\node (4) at (8,3.5) {$\mathcal{L_B}(X)$};
				\node (7) at (12,0) {$\mathcal{L_B}(K_\phi)$};
				\node (0) at (12,3.5) {$\mathcal{L_B}(D_1)$};
				\footnotesize
				\draw[arr] (0) to 
				["$\rho_\phi$"] (7);
				\draw[arr]   (3) to ["$\tilde{\pi}_\phi$"]    (4);
				\draw[arr]   (4) to ["$\gamma_\phi$"]    (0);
				\draw[arr]   (3) to ["$\pi_\phi$"]    (7);

			\end{tikzpicture}
			
		\end{figure}

		For $a \in \mathcal{A} \text{ and }d \in D,$ we have 
		\begin{equation*}
			\begin{split}
				{V_\phi}^*\pi_\phi(a)V_\phi(d) 
				&= {\eta(\xi \otimes 1_\mathcal{B})}^*\rho_\phi \circ\gamma_\phi\circ\tilde{\pi}_\phi(a)\eta(\xi \otimes1_\mathcal{B})(d) \\
				&={\eta(\xi \otimes 1_\mathcal{B})}^*\gamma_\phi\circ\tilde{\pi}_\phi(a)(\xi \otimes 1_\mathcal{B}) \otimes d \\
				&={\eta(\xi \otimes 1_\mathcal{B})}^*\tilde{\pi}_\phi(a)\xi \otimes 1_\mathcal{B} \otimes d \\
				&= \langle \xi \otimes 1_\mathcal{B},\tilde{\pi}_\phi(a)\xi \otimes 1_\mathcal{B}\rangle d\\
					&= \langle \xi,\tilde{\pi}_\phi(a)\xi \rangle d\\
				&= \phi(a)d.
			\end{split}
		\end{equation*}	  
		
		Lastly, we show the minimality of this representation. To do this, first we show that $xb\otimes b'=x\otimes bb'$, whenever $x \in X$ and $b,b' \in \mathcal{B}$.  So, let  $w = (x b)\otimes b' - x \otimes (b b')$, where $x \in X$ and $b,b' \in \mathcal{B}$. 
		Then
		\[
		\begin{aligned}
			\langle (x b)\otimes b', (x b)\otimes b' \rangle 
			&= (b')^* \,\langle x b, x b \rangle \, b' 
			= (b')^* \, b^* \langle x, x \rangle b \, b',\\[6pt]
			\langle x \otimes (b b'), x \otimes (b b') \rangle 
			&= (b b')^* \langle x, x \rangle (b b') 
			= (b')^* b^* \langle x, x \rangle b\, b',\\[6pt]
			\langle (x b) \otimes b',\, x \otimes (b b') \rangle 
			&= (b')^* \langle x b, x \rangle (b b') 
			= (b')^* b^* \langle x, x \rangle b\, b',\\[6pt]
			\langle x \otimes (b b'),\, (x b) \otimes b' \rangle 
			&= (b b')^* \langle x, x b \rangle b'
			= (b')^* b^* \langle x, x \rangle b\, b'.
		\end{aligned}
		\]
		
		Hence 
		\begin{align*}
			\langle w, w \rangle 
			&= (b')^* b^* \langle x, x \rangle b\, b' 
			+ (b')^* b^* \langle x, x \rangle b\, b' 
			- (b')^* b^* \langle x, x \rangle b\, b' 
			- (b')^* b^* \langle x, x \rangle b\, b' \\
			&= 0.
		\end{align*}
		Therefore  
$(x b)\otimes b' = x \otimes (b b')$ in the quotient module.
		Put $x= \tilde{\pi}_\phi(a)\xi$ in the above calculations to obtain $( \tilde{\pi}_\phi(a)\xi b)\otimes b' =  \tilde{\pi}_\phi(a)\xi \otimes (b b')$.

Next, let us examine how elements of the algebra $\mathcal{B}$ interact within the tensor product and the inner product. For $d_1, {d_1}' \in D_1$ and $d_2 = \xi \otimes b_2, {d_2}'=\xi \otimes {b_2}' \in D$ and $b \in \mathcal{B}$, we have
\begin{align*}
    \langle d_1 \otimes \xi \otimes b_2, {d_1}' \otimes \xi \otimes {b_2}'b  \rangle' 
    &= \langle \xi \otimes b_2, \tau(\langle d_1, {d_1}' \rangle) (\xi \otimes {b_2}'b)  \rangle \\
    &= \langle \xi \otimes b_2,  \xi \otimes \langle d_1, {d_1}' \rangle{b_2}'b \rangle \\
    &= \langle \xi \otimes b_2,  \xi \otimes \langle d_1, {d_1}'{b_2}' \rangle b \rangle \\
    &=  \langle d_1 \otimes \xi \otimes b_2, {d_1}'{b_2}' \otimes \xi \otimes b  \rangle'.
\end{align*}
The equality ${d_1}' \otimes \xi \otimes {b_2}'b = {d_1}'{b_2}' \otimes \xi \otimes b $ allows the factor ${b_2}'$  to be absorbed into the first tensor component, which is central to describing the full $\mathcal{B}$-span.

		Finally, combining this with the fact that $[\tilde{\pi}_\phi(\mathcal{A})\xi]_\mathcal{B} = X$, we obtain
		$$[\pi_\phi(\mathcal{A})V_\phi (D)]_\mathcal{B} = [\tilde{\pi}_\phi(\mathcal{A})\xi \otimes 1_\mathcal{B} \otimes D]_\mathcal{B}= K_\phi.$$  

	\end{proof}

	\begin{remark}
		\begin{enumerate}
		    \item In every instance, $[\cdot]_\mathcal{B}$ represents the closed linear span over the algebra $\mathcal{B}$ taken on the right.
             \item In Theorem \ref{prorep1}, the subscript $\otimes_\tau$ indicates that the inner product is defined via $\tau$. The same approach applies to other maps, and once the structure is fixed, we drop the subscript and write simply $\otimes$.
		\end{enumerate}

	\end{remark}
	
	We note the associative nature of tensor products with respect to different inner products, which will be useful in the proofs that follow and is of independent interest.

	\begin{remark}\label{asso}
		Let $E$ be Hilbert module over a  pro-$C^*$-algebra $\mathcal{A}$. Suppose $D, \tilde{\pi}_\phi, \rho_\phi, \gamma_\phi, \pi_\phi$  are as in Theorem \ref{prorep1}. Then $(E \otimes_{\gamma_\phi \circ \tilde{\pi}_\phi} D_1) \otimes_\tau D = E \otimes_{\pi_\phi} (D_1 \otimes_\tau D). $ 
		
		Indeed, for $z_1,z_2 \in E,$ $x_1,x_2 \in D_1 \text{ and } y_1,y_2 \in D,$ we have
		\begin{align*}
			\left\langle (z_1 \otimes x_1) \otimes y_1, (z_2 \otimes x_2) \otimes y_2\right\rangle &= \left\langle y_1,\langle z_1 \otimes x_1 ,z_2 \otimes x_2 \rangle y_2\right\rangle \\
			&=\left\langle y_1,\left\langle x_1 , \gamma_\phi \circ \tilde{\pi}_\phi(\langle z_1,z_2 \rangle ) x_2 \right\rangle y_2\right\rangle \\
			&=\left\langle x_1 \otimes y_1, \left( \gamma_\phi \circ \tilde{\pi}_\phi(\langle z_1,z_2 \rangle ) x_2 \right) \otimes y_2 \right \rangle' \\
			&=\left\langle x_1 \otimes y_1,  \rho_\phi\circ\gamma_\phi \circ \tilde{\pi}_\phi(\langle z_1,z_2 \rangle )( x_2 \otimes y_2)\right \rangle' \\
			&=\left\langle x_1 \otimes y_1,  \pi_{\phi}(\langle z_1,z_2 \rangle )( x_2 \otimes y_2)\right \rangle' \\
			&= \left\langle z_1 \otimes (x_1 \otimes y_1), z_2 \otimes (x_2 \otimes y_2)\right\rangle.
		\end{align*}
	\end{remark}

 	The following lemma 
	a completely positive map in terms of an isometry. The inspiration is drawn from \cite[Lemma 1]{KS}.
	
	\begin{lemma}\label{proiso}
		Let $\mathcal{A}$ and $\mathcal{B}$ be pro-$C^*$-algebras. Let $E$ be Hilbert $\mathcal{A}$-module and $F$ be a Hilbert $\mathcal{B}$-module. Let $\phi:\mathcal{A} \rightarrow \mathcal{B}$ be a continuous, unital, completely positive map and $\Phi: E \rightarrow F$ be a $\phi$-map. Suppose $D, K_\phi, \xi, \tilde{\pi}_\phi, \pi_\phi \text{ and } \gamma_\phi $ are same as in Theorem \ref{prorep1} and $E \otimes_{\pi_\phi} K_\phi$ is associated with $(E \otimes_{\gamma_\phi\circ\tilde{\pi}_\phi} D_1) \otimes_\tau D $. Then there exists an isometry $v: E \otimes_{\pi_\phi} K_\phi \rightarrow [\Phi(E)]_\mathcal{B}$ such that 
		\begin{equation*}
			v\left(z \otimes (\xi \otimes 1_\mathcal{B}) \otimes (\xi \otimes 1_\mathcal{B})\right) = \Phi(z),
		\end{equation*}
	for all $z \in E$.
	\end{lemma}

	\begin{proof}
		Let $\chi_\phi$ denote the set span$\{\tilde{\pi}_\phi(a)(\xi b): a \in \mathcal{A}, b \in \mathcal{B}\}.$		
		Observe that $E \otimes_{\gamma_\phi\circ\tilde{\pi}_\phi} \left(D_1 \otimes_\tau D\right )$ is a right $\mathcal{B}$-module. For $e \in E, k \in K_\phi$, the right action of $\mathcal{B}$ on $E \otimes_{\pi_\phi}K_\phi$ is given by $(e \otimes k)b = e \otimes kb$, for all $b \in \mathcal{B}.$
		
		For any $e_1,e_2 \in E , \text{ and } k_1,k_2 \in \left(D_1 \otimes_\tau D \right),$ the inner-product on $\left(E \otimes_{{\pi}_\phi} \left(D_1 \otimes_\tau D\right)\right)$ is defined by 
		\begin{equation*}
			\left\langle e_1 \otimes k_1 , e_2 \otimes k_2 \right\rangle := \left\langle k_1, {\pi}_\phi(\langle e_1, e_2 \rangle)k_2 \right\rangle'.
		\end{equation*}

		Define $v_0: \left(E \otimes_{\pi_\phi}K_\phi\right) \rightarrow [\Phi(E)]_\mathcal{B}$ by 
		\begin{equation*}
			v_0\left( z \otimes \left(\tilde{\pi}_\phi(a)(\xi) \otimes b_0 \right) \otimes (\xi \otimes b_1) \right) = \Phi(za)b_0 b_1, 
		\end{equation*}
		for all $a \in \mathcal{A}, b_0, b_1 \in \mathcal{B}$ and $z \in E.$

       For \(a, c \in \mathcal{A}\), \(b_0,b_1,d_0,d_1 \in \mathcal{B}\) and \(z,w\in E\), observe
\begin{align*}
&\Big\langle z\otimes\big(\tilde{\pi}_\phi(a)(\xi)\otimes b_0\big)\otimes(\xi\otimes b_1),\;
w\otimes\big(\tilde{\pi}_\phi(c)(\xi )\otimes d_0\big)\otimes(\xi\otimes d_1)\Big\rangle \\[6pt]
&\;=\;\Big\langle \big(\tilde{\pi}_\phi(a)(\xi)\otimes b_0\big)\otimes(\xi\otimes b_1),\;
\rho_\phi\circ\gamma_\phi\circ\tilde{\pi}_\phi(\langle z,w\rangle)\big(\tilde{\pi}_\phi(c)(\xi)\otimes d_0\big)\otimes(\xi\otimes d_1)\Big\rangle' \\[6pt]
&\;=\;\Big\langle \big(\tilde{\pi}_\phi(a)(\xi)\otimes b_0\big)\otimes(\xi\otimes b_1),\;
\big(\tilde{\pi}_\phi(\langle z,w\rangle)\tilde{\pi}_\phi(c)(\xi)\otimes d_0\big)\otimes(\xi\otimes d_1)\Big\rangle' \\[6pt]
&\;=\;\Big\langle \xi\otimes b_1,\;
\big\langle \tilde{\pi}_\phi(a)(\xi)\otimes b_0, \tilde{\pi}_\phi(\langle z,wc\rangle)(\xi)\otimes d_0 \big\rangle (\xi\otimes d_1)\Big\rangle \\[6pt]
&\;=\;\Big\langle \xi\otimes b_1,\;
{ b_0}^*\big\langle \tilde{\pi}_\phi(a)(\xi), \tilde{\pi}_\phi(\langle z,wc\rangle)(\xi) \big\rangle d_0  (\xi\otimes d_1)\Big\rangle \\[6pt]
&\;=\;\Big\langle \xi\otimes b_1,\;
 \xi\otimes { b_0}^*\big\langle \xi, \tilde{\pi}_\phi(\langle za,wc\rangle)(\xi) \big\rangle d_0 d_1\Big\rangle \\[6pt]
 &\;=\; {b_1}^*\big\langle \xi,\;
 \xi \big\rangle { b_0}^*\phi(\langle za,wc\rangle)d_0 d_1 \\[6pt]
 &\;=\; {b_1}^* { b_0}^*\big \langle \Phi( za), \Phi(wc) \big \rangle d_0 d_1 \\[6pt]
  &\;=\; \big \langle \Phi( za)b_0 b_1, \Phi(wc)d_0 d_1 \big \rangle.
\end{align*}

		Since 
		$ \text{span}\left\{\left(\tilde{\pi}_\phi(a)(\xi) \otimes b_1 \right) \otimes (\xi \otimes b_2)  \ : \ a \in \mathcal{A}, b_1,b_2 \in \mathcal{B} \right\}$
		is dense in $K_\phi$, the map $v_0$ can be uniquely extended to  an isometry $v: E \otimes_{\pi_\phi} K_\phi \rightarrow [\Phi(E)]_\mathcal{B}.$ In particular,	$v\left(z \otimes (\xi \otimes 1_\mathcal{B}) \otimes (\xi \otimes 1_\mathcal{B}) \right) = \Phi(z)$.

        We can observe that $v$ is adjointable. Indeed, the adjoint
		of $v$ is given on the dense set of generators by
		\[
		v^*(\Phi(za)b) = z \otimes (\tilde{\pi}_\phi(a)\xi \otimes 1_\mathcal{B}) \otimes (\xi \otimes b),
		\qquad z \in E,\; b \in \mathcal{B}.
		\]
	\end{proof}
	
	
	The following theorem is a modified version of \cite[Theorem 4]{KS} applied to completely positive maps between Hilbert modules over pro-$C^*$-algebras.

			

\begin{theorem}\label{prorep2}
		Let $\mathcal{A}, \mathcal{B}$ be  pro-$C^*$-algebras and $\phi: \mathcal{A} \rightarrow \mathcal{B}$ be a continuous, unital, completely positive map. Let $E$ be a Hilbert $\mathcal{A}$-module, $F$ be a Hilbert $\mathcal{B}$-module and $\Phi: E \rightarrow F$ be a $\phi$-map. Then there exist Hilbert $\mathcal{B}$-modules $D$ and $X,$ with an element $\xi \in X$, and triples $\left(\pi_\phi,V_\phi,K_\phi\right)$ and $\left(\pi_\Phi,W_\Phi,K_\Phi\right)$ such that 
		\begin{enumerate}
			
			\item $K_\phi$ and $K_\Phi$ are  Hilbert 
            $\mathcal{B}$-modules;
			\item$\pi_\phi: \mathcal{A} \rightarrow \mathcal{L}_\mathcal{B}(K_\phi)$ is a unital representation of $\mathcal{A}$;
			\item $\pi_\Phi: E \rightarrow \mathcal{L}_\mathcal{B}(K_\phi,K_\Phi)$ is a $\pi_\phi$-morphism;
			\item $V_\phi: D \rightarrow K_\phi$ and $W_\Phi: [\Phi(E)]_\mathcal{B} \rightarrow K_\Phi$ are bounded linear operators such that 
			$$\phi(a)I_D = {V_\phi}^*\pi_\phi(a)V_\phi$$ for all $a \in \mathcal{A}$, and
			$$\Phi(z) = {W_\Phi}^*\pi_\Phi(z)V_\phi (\xi \otimes 1_\mathcal{B})$$ for all $z \in E.$
		\end{enumerate}
	\end{theorem}

	\begin{proof}
		Let $\tilde{\pi}_\phi$ be the Paschke's GNS representation associated with $\phi$ as specified in Theorem \ref{Paschpro}, that is, we get a Hilbert $\mathcal{B}$-module $X$, an element $\xi \in X$ and a unital continuous representation $\tilde{\pi}_\phi : \mathcal{A} \rightarrow \mathcal{L_B}(X)$ such that $\phi(a) = \langle\xi, \tilde{\pi}_\phi(a)\xi\rangle,$ for all $a \in \mathcal{A}.$

		Let $\left(\pi_\phi, V_\phi, K_\phi, D\right)$ be the Stinespring's representation for $\phi$ obtained as in Theorem \ref{prorep1}.
		Indeed, we obtain a Hilbert $\mathcal{B}$-module $K_\phi$ associated with $D_1 \otimes_\tau D$, a bounded operator $V_\phi \in \mathcal{L}_\mathcal{B}(D,K_\phi)$ (where $V_\phi := \eta(\xi \otimes 1_\mathcal{B})$) and a unital representation $\pi_\phi: \mathcal{A} \rightarrow \mathcal{L}_\mathcal{B}(K_\phi)$, where $\pi_\phi:=\rho_\phi \circ\gamma_\phi\circ\tilde{\pi}_\phi.$ 
		
		Let $K_\Phi$ be the Hilbert $\mathcal{B}$-module associated with $E \otimes_{\pi_\phi} K_\phi$. 
		Now, for $z \in E$, define $T_z: K_\phi\rightarrow E \otimes_{\pi_\phi} K_\phi$  by 
		\begin{equation*}
			T_z(k) = z \otimes k,\; \text{for all}\; k \in K_\phi.
		\end{equation*}
		Clearly, $T_z$ is $\mathcal{B}$-linear, in fact $T_z \in \mathcal{L}_\mathcal{B}(K_\phi, K_\Phi).$ Indeed, for  $q \in S(\mathcal{B}),$ we have
		\begin{equation*}
			\begin{split}
				\|T_z(k)\|_q^2 
				&= q(\langle z \otimes k, z \otimes k \rangle) \\ 
				&= q(\langle  k, \pi_\phi(\langle z,z \rangle) k \rangle') \\ 
				& \leq K {q(\langle k, k \rangle')} = K {\|k\|_q}^2.
			\end{split}
		\end{equation*}
		Also, for $z,z_0 \in E,\text{ and } k,k_0 \in K_\phi,$
		\begin{equation*}
			\begin{split}
				\langle  z_0 \otimes k_0, T_z (k) \rangle &= \langle k_0,\pi_\phi(\langle z_0 , z \rangle) k \rangle'\\
				&= \langle \pi_\phi(\langle z , z_0 \rangle) k_0, k \rangle'\\
				&= \langle {T_z}^*(z_0 \otimes k_0), k  \rangle.
			\end{split}
		\end{equation*}
		We get ${T_z}^*(z_0 \otimes k_0) = \pi_\phi(\langle z , z_0 \rangle) k_0$, and hence  $T_z \in \mathcal{L_B}(K_\phi, K_\Phi)$.

		Define $\pi_\Phi: E \rightarrow \mathcal{L}_\mathcal{B}(K_\phi, K_\Phi)$ by $$\pi_\Phi(z) = T_z,$$ for all $z \in E.$ 
		Observe that $\pi_\Phi$ is a $\pi_\phi$-morphism.  Indeed, for $z, w \in E$ and $k \in K_\phi,$ 
		$$\langle \pi_\Phi(z), \pi_\Phi(w) \rangle k =  \pi_\Phi(z)^*\pi_\Phi(w) k = T_z^*T_w k =T_z^*(w \otimes k) =\pi_\phi(\langle z,w \rangle) k.$$		
		
		Furthermore, by Lemma \ref{proiso}, there exists an isometry $v: E \otimes_{\pi_\phi} K_\phi \rightarrow [\Phi(E)]_\mathcal{B}$ such that $v\left(z \otimes (\xi \otimes 1_\mathcal{B}) \otimes (\xi \otimes 1_\mathcal{B})\right) = \Phi(z).$

		Let $W_\Phi:=v^*.$ For $z \in E \text{ and } d = (\xi \otimes b) \in D,$ we have
		\begin{align*}
			{W_\Phi}^*\pi_\Phi(z)V_\phi(d) &= v \left( T_z (\xi \otimes 1_\mathcal{B} \otimes d)\right) \\
			&= v\left( z \otimes \xi \otimes 1_\mathcal{B} \otimes d\right) \\
			&= \Phi(z)b.
		\end{align*}
		In particular, for $b = 1_\mathcal{B},$ then
		\begin{equation*}
			\Phi(z) = 	{W_\Phi}^*\pi_\Phi(z)V_\phi(\xi \otimes 1_\mathcal{B}),
		\end{equation*}
		for all $z \in E.$
	\end{proof}

	\begin{remark}
		Let $\phi$ and $\Phi$ be as in Theorem \ref{prorep2}.
		\begin{enumerate}

			\item The bounded linear operator $W_\Phi: [\Phi(E)]_\mathcal{B} \rightarrow K_\Phi$ is a coisometry. 
			
			\item If $E$ is full, that is $\langle E, E\rangle = \mathcal{A}$, then $\pi_\Phi: E \rightarrow \mathcal{L}_{\mathcal{B}}(K_\phi,K_\Phi)$ is a non-degenerate representation of $E$.
			Indeed, for $z \in E$ and $k_\phi \in K_\phi$, following the definition, we have, 
			$\pi_{\Phi}(z)(k_\phi) = z \otimes k_\phi$. Hence, $[\pi_\Phi(E)(K_\phi)]_\mathcal{B}= K_\Phi.$ 
			Since $\langle E, E\rangle = \mathcal{A},$ we have 
			\begin{align*}
				[{\pi_\Phi(E)}^*(K_\Phi)]_\mathcal{B} &= [{T_E}^*(K_\Phi)]_\mathcal{B} \\
				&= [\pi_\phi(\langle E, E\rangle) (K_\phi)]_\mathcal{B} \\
				&= [\tilde{\pi}_\phi(\langle E, E\rangle)X \otimes \mathcal{B} \otimes \xi \otimes \mathcal{B}  ]_\mathcal{B} \\
				&= K_\phi.
			\end{align*}

		\end{enumerate}

	\end{remark}

	\begin{definition}
		Let $\phi$ and $\Phi$ be as in Theorem \ref{prorep2}. We call the pair, $\{(\pi_{\phi}, V_\phi, K_\phi), (\pi_\Phi, W_\Phi, K_\Phi)\}$ a Stinespring's representation of $(\phi, \Phi)$ if the conditions $(1)-(4)$ of Theorem \ref{prorep2} are satisfied. This representation is said to be minimal if 
		\begin{enumerate}
			\item $[\pi_\phi(\mathcal{A})V_\phi( D)]_\mathcal{B} = K_\phi,$ and 
			\item $[\pi_\Phi(E)V_\phi (D)]_\mathcal{B} = K_\Phi.$
		\end{enumerate}
	\end{definition}

	\begin{remark}
		Observe that the pair $\{(\pi_{\phi}, V_\phi, K_\phi), (\pi_\Phi, W_\Phi, K_\Phi)\}$ obtained in Theorem \ref{prorep2} is a minimal representation of $(\phi,\Phi).$ Indeed, we have seen in Theorem \ref{prorep1}
		\begin{align*}
			[\pi_{\phi}(\mathcal{A})V_\phi (D)]_\mathcal{B} &= [(\rho_\phi \circ\gamma_\phi\circ\tilde{\pi}_\phi)(\mathcal{A})\eta( \xi \otimes 1_\mathcal{B})(D)]_\mathcal{B} \\
			&= [ \tilde{\pi}_\phi(\mathcal{A})\xi \otimes 1_\mathcal{B} \otimes D ]_\mathcal{B} \\
			&= K_\phi. 
		\end{align*}

Observe that for $e \in E, d \in D$ and $b \in \mathcal{B}$,
		\[ ea  \otimes (x \otimes b) \otimes d \;=\; e  \otimes (\tilde{\pi}_\phi(a)x \otimes b) \otimes d, 
		\qquad a\in \mathcal{A}.
		\]
To establish this identity, we compute inner products against the elementary tensors
$e'\otimes(x'\otimes b')\otimes d'$, which form a dense submodule of 
$E \otimes_{\pi_\phi} K_\phi$. For an arbitrary such element, we obtain
        \begin{align*}
\big\langle ea\otimes(x\otimes b)\otimes d,\; e'\otimes(x'\otimes b')\otimes d'\big\rangle 
&= \big\langle (x\otimes b)\otimes d,\; \pi_\phi(\langle ea,e'\rangle)\,(x'\otimes b')\otimes d'\big\rangle'
\\[6pt]
&= \big\langle (x\otimes b)\otimes d,\; \pi_\phi(a^*\langle e,e'\rangle)\,(x'\otimes b')\otimes d'\big\rangle'\\[6pt]
&= \big\langle (x\otimes b)\otimes d,\; \pi_\phi(a^*)\,\pi_\phi(\langle e,e'\rangle)\,(x'\otimes b')\otimes d'\big\rangle'
\\[6pt]
&= \big\langle \pi_\phi(a)\big((x\otimes b)\otimes d\big),\; \pi_\phi(\langle e,e'\rangle)\,(x'\otimes b')\otimes d'\big\rangle'
\\[6pt]
&= \big\langle \big(\tilde{\pi}_\phi(a)x\otimes b\big)\otimes d,\; \pi_\phi(\langle e,e'\rangle)\,(x'\otimes b')\otimes d'\big\rangle'
\\[6pt]
&= \big\langle e\otimes(\tilde{\pi}_\phi(a)x\otimes b)\otimes d,\; e'\otimes(x'\otimes b')\otimes d'\big\rangle
\end{align*}
        
     Consequently, applying the preceding observation regarding the action within the tensor product, we have
		\begin{align*}
		[\pi_{\Phi}(E)V_\phi (D)]_\mathcal{B} &= [T_E\left(\xi \otimes 1_\mathcal{B} \otimes D) \right)]_\mathcal{B} \\
			&= [E \otimes  \xi \otimes 1_\mathcal{B} \otimes D]_\mathcal{B} \\
			&= E \otimes K_\phi \\
			&= K_\Phi. 
		\end{align*}

	\end{remark}

	We now show that the uniqueness of the minimal Stinespring's representation is up to the unitary equivalence.
	
	\begin{theorem}\label{uniequi}
		Let $\phi$ and $\Phi$ admit the minimal Stinespring representations 
		$(\pi_\phi, V_\phi, K_\phi)$ and $(\pi_\Phi, W_\Phi, K_\Phi)$, respectively. 
		Suppose further that $\{(\pi_{\mathcal{A}}, V', K'), (\pi_E, W', \tilde{K})\}$ 
		is another minimal Stinespring representation associated with $(\phi,\Phi)$. Then there are two unitary operators $U_1: K_\phi \rightarrow K^\prime$ and $U_2: K_\Phi \rightarrow \tilde{K}$ such that 
		\begin{enumerate}
			\item $V^\prime = U_1V_\phi$.
			\item $U_1\pi_{\phi}(a) = \pi_{\mathcal{A}}(a)U_1,$ for all $a \in \mathcal{A}.$
			\item $U_2\pi_{\Phi}(z) = \pi_{E}(z)U_1,$ for all $z \in E.$
			\item $W^\prime = U_2W_\Phi$. 
		\end{enumerate}
	\end{theorem}

	\begin{proof}
		It is easy to see the structure of the maps through the following figure.		
		
		\tikzstyle{line} = [draw, -latex'] 
		
		\begin{figure}[hbt!]
			\begin{tikzpicture}[scale=0.60, arr/.style = {-Stealth}]
				\node (1) at (0,3.5) {$D$};
				\node (2) at (4,3.5) {$K_{\phi}$};
				\node (3) at (8,3.5) {$K_\phi$};
				\node (4) at (12,3.5) {$K_\Phi$};
				\node (5) at (16,3.5) {$[\Phi(E)]_\mathcal{B}$};
				\node (6) at (4,0) {$K^\prime$};
				\node (7) at (8,0) {$K^\prime$};
				\node (8) at (12,0) {$\tilde{K}$};
				\footnotesize
				\draw[arr]   (1) to ["$V_\phi$"]    (2);
				\draw[arr]  (2) to ["$\pi_{\phi}(a)$"]    (3);
				\draw[arr]   (3) to ["$\pi_\Phi(z)$"]    (4);
				\draw[arr]   (5) to ["$W_\Phi$"]    (4);
				\draw[arr]   (1) to ["$V^\prime$"]    (6);
				\draw[arr]   (6) to ["$\pi_{\mathcal{A}}(a)$"]    (7);
				\draw[arr]   (7) to ["$\pi_E(z)$"]    (8);
				\draw[arr]   (5) to ["$W^\prime$"]    (8);
				\draw[arr]   (2) to ["$U_1$"]    (6);
				\draw[arr]   (4) to ["$U_2$"]    (8);
				\draw[arr]   (3) to ["$U_1$"]    (7);

			\end{tikzpicture}
			
		\end{figure}

		Define $U_1: K_\phi \rightarrow K^\prime $ as follows
		$$U_1\left(\sum_{i=1}^n\pi_{\phi}(a_i)V_\phi d_i\right) = \sum_{i=1}^{n} \pi_{\mathcal{A}}(a_i)V^\prime d_i,$$
		for all $a_1,\dots a_n \in \mathcal{A}$ and $d_1, \dots, d_n \in D, n \geq 1.$ Since $[\pi_\phi(\mathcal{A})V_\phi D]_\mathcal{B} = K_\phi,$ we can uniquely extend this operator to an operator from $K_\phi$ to $K^\prime$.
		For any $d \in D$, substitute $a = 1_\mathcal{A} $ in $U_1(\pi_{\phi}(a)V_\phi d) = \pi_\mathcal{A}(a)V^\prime d$ to obtain $V^\prime = U_1V_\phi$.

       Note that $U_1$ is adjointable, and its adjoint is determined on elementary vectors by
\[
{U_1}^{*}\!\left(\pi_{\mathcal{A}}(a)V' d\right)
   = \pi_{\phi}(a)\,V_{\phi} d,
   \qquad a \in \mathcal{A},\; d \in D.
\]

		Moreover, for $a_1,\dots a_n \in \mathcal{A}$ and $d_1, \dots, d_n \in D,$ we have 
		\begin{align*}
				\left\langle U_1\left(\sum_{i=1}^n\pi_{\phi}(a_i)V_\phi d_i\right), U_1\left(\sum_{i=1}^n\pi_{\phi}(a_i)V_\phi d_i\right) \right \rangle &=  \sum_{i,j=1}^n \big \langle \pi_{\mathcal{A}}(a_i)V^\prime d_i, \pi_{\mathcal{A}}(a_j)V^\prime d_j \big\rangle   \\
				&= \sum_{i,j=1}^n \big \langle V^\prime d_i, \pi_{\mathcal{A}}({a_i}^*a_j)V^\prime d_j \big\rangle   \\
				&= \sum_{i,j=1}^n \big \langle  d_i, {V^\prime}^*\pi_{\mathcal{A}}({a_i}^*a_j)V^\prime d_j \big\rangle   \\
				&= \sum_{i,j=1}^n \big \langle  d_i, \phi({a_i}^*a_j) d_j \big\rangle   \\
				&= \sum_{i,j=1}^n \big \langle  d_i, {V_\phi}^*\pi_\phi({a_i}^*a_j)V_\phi d_j \big\rangle   \\
				&= \sum_{i,j=1}^n \big \langle  \pi_\phi(a_i)V_\phi d_i, \pi_\phi(a_j)V_\phi d_j \big\rangle'   \\
				&=  \left\langle \sum_{i=1}^n\pi_{\phi}(a_i)V_\phi d_i, \sum_{i=1}^n\pi_{\phi}(a_i)V_\phi d_i \right \rangle'.
		\end{align*}
		This shows that $U_1$ is an isometry. 
		
		Next, define $U_2: K_\Phi \rightarrow \tilde{K} $ by
		\begin{equation*}
			U_2\left(\sum_{i=1}^n\pi_{\Phi}(z_i)V_\phi d_i\right) = \sum_{i=1}^{n} \pi_{E}(z_i)V^\prime d_i,
		\end{equation*}
		for all $z_1,\dots z_n \in E$ and $d_1, \dots, d_n \in D, n \geq 1.$ Again, since $[\pi_\Phi(E)V_\phi D]_\mathcal{B} = K_\Phi,$ we can uniquely extend this operator to an operator from $K_\Phi$ to $\tilde{K}$. 

Note that $U_2$ is adjointable, and its adjoint is determined on elementary vectors by
\[
{U_2}^{*}\!\left(\pi_E(z)V' d\right)
   = \pi_{\Phi}(z)\,V_{\phi} d,
   \qquad z \in E,\; d \in D.
\]

        Using the fact that $\pi_{\Phi}$ is a $\pi_{\phi}$-morphism and $\pi_E$ is a $\pi_{\mathcal{A}}$-morphism, we observe that $U_2$ is an isometry.
		Indeed, for $z_1,\dots z_n \in E$ and $d_1, \dots, d_n \in D,$ we have

		\begin{align*}
				\left\langle U_2\left(\sum_{i=1}^n\pi_{\Phi}(z_i)V_\phi d_i\right), U_2\left(\sum_{i=1}^n\pi_{\Phi}(z_i)V_\phi d_i\right) \right \rangle &=  \sum_{i,j=1}^n \big \langle \pi_{E}(z_i)V^\prime d_i, \pi_{E}(z_j)V^\prime d_j \big\rangle   \\
				&= \sum_{i,j=1}^n \big \langle V^\prime d_i, \pi_E(z_i)^*\pi_{E}(z_j)V^\prime d_j \big\rangle   \\
				&= \sum_{i,j=1}^n \big \langle  d_i, {V^\prime}^*\langle \pi_E(z_i), \pi_{E}(z_j) \rangle V^\prime d_j \big\rangle   \\
				&= \sum_{i,j=1}^n \big \langle  d_i, {V^\prime}^* \pi_\mathcal{A}(\langle z_i, z_j \rangle) V^\prime d_j \big\rangle   \\
				&= \sum_{i,j=1}^n \big \langle  d_i, {V_\phi}^* \pi_\phi(\langle z_i, z_j \rangle) V_\phi d_j \big\rangle   \\
				&= \sum_{i,j=1}^n \big \langle  d_i, {V^\prime}^*\langle \pi_\Phi(z_i), \pi_\Phi(z_j) \rangle V^\prime d_j \big\rangle   \\		
				&= \sum_{i,j=1}^n \big \langle V^\prime d_i, \pi_\Phi(z_i)^*\pi_\Phi(z_j)V^\prime d_j \big\rangle   \\	
				&=  \left\langle \sum_{i=1}^n\pi_\Phi(z_i)V_\phi d_i, \sum_{i=1}^n\pi_\Phi(z_i)V_\phi d_i \right \rangle .	
		\end{align*}
		
		Thus, it can be extended to a unitary $U_2$ from $K_\Phi$ to $\tilde{K}$.
		
		Furthermore, using $(2)$, for $z \in E, a_1 \dots a_n \in \mathcal{A}$ and $d_1, \dots, d_n \in D,$ we have
		\begin{equation*}
			\begin{split}
				&	\left\langle \pi_{E}(z) U_1\left(\sum_{i=1}^n\pi_{\phi}(a_i)V_\phi d_i\right),\pi_E(z) U_1\left(\sum_{i=1}^n\pi_{\phi}(a_i)V_\phi d_i\right)\right\rangle \\
				&= \left \langle U_1\left(\sum_{i=1}^n\pi_{\phi}(a_i)V_\phi d_i\right), \left\langle \pi_E(z), \pi_E(z) \right\rangle U_1\left(\sum_{i=1}^n\pi_{\phi}(a_i)V_\phi d_i\right)\right\rangle  \\
				&=\left \langle \sum_{i=1}^n\pi_{\phi}(a_i)V_\phi d_i, {U_1}^*  \pi_\mathcal{A}( \langle z, z \rangle) U_1\left(\sum_{i=1}^n\pi_{\phi}(a_i)V_\phi d_i\right)\right\rangle'  \\
				&= \left \langle \sum_{i=1}^n\pi_{\phi}(a_i)V_\phi d_i, \pi_\phi( \langle z, z \rangle) \left(\sum_{i=1}^n\pi_{\phi}(a_i)V_\phi d_i\right)\right\rangle'  \\
				&= \left \langle \sum_{i=1}^n\pi_{\phi}(a_i)V_\phi d_i, \left\langle \pi_\Phi(z), \pi_\Phi(z) \right\rangle \left(\sum_{i=1}^n\pi_{\phi}(a_i)V_\phi d_i\right) \right\rangle'  \\
				&= \left \langle \sum_{i=1}^n\pi_{\phi}(a_i)V_\phi d_i,  {\pi_\Phi(z)}^*{U_2}^*U_2\pi_\Phi(z)  \left(\sum_{i=1}^n\pi_{\phi}(a_i)V_\phi d_i\right) \right\rangle'  \\
				&= \left \langle  U_2\pi_\Phi(z) \left(\sum_{i=1}^n\pi_{\phi}(a_i)V_\phi d_i\right), U_2\pi_\Phi(z)  \left(\sum_{i=1}^n\pi_{\phi}(a_i)V_\phi d_i \right) \right\rangle.	
			\end{split}
		\end{equation*}
		Therefore, $U_2\pi_{\Phi}(z) = \pi_{E}(z)U_1,$ for all $z \in E.$

		Moreover, using $(1)$ and $(3)$, for $z \in E, a \in \mathcal{A}$ and $d =  \xi \otimes b \in D,$ we have
		\begin{equation*}
			\begin{split}
				\Phi(z)b
				&= {W^\prime}^*\pi_{E}(z)V^\prime\left(d\right) \\
				&=  {W^\prime}^*U_2\pi_{\Phi}(z){U_1}^*V^\prime\left( d\right) \\
				&={W^\prime}^*U_2\pi_{\Phi}(z)V_\phi\left(d\right), \text{ and } \\
				\Phi(z)b
				&= {W_\Phi}^*\pi_{\Phi}(z)V_\phi\left(d\right). 
			\end{split}
		\end{equation*}

		This implies, $\left({W^\prime}^*U_2 - {W_\Phi}^*\right)\pi_{\Phi}(z)V_\phi\left(d\right) = 0.$ Since $[\pi_{\Phi}(E)V_\phi D]_\mathcal{B} = K_\Phi$, we obtain ${W^\prime}^*U_2 = {W_\Phi}^*.$

	\end{proof}	
	
	\begin{remark}
		In the definitions of the operators \(U_1\) and \(U_2\) in Theorem \ref{uniequi}, the linear span is
		understood to be taken over \(\mathcal{B}\) on the
		right; these algebra elements are absorbed into the representation, so they are
		not written explicitly.
	\end{remark}

	As a consequence of Theorem \ref{uniequi}, we extract the following corollary, which stands alone as a separate result.
	
	\begin{corollary}
		Let $\mathcal{A}, \mathcal{B}$ be  pro-$C^*$-algebras and $\phi: \mathcal{A} \rightarrow \mathcal{B}$ be a continuous, unital, completely positive map. Let $(\pi_\phi,V_\phi,K_\phi)$ be obtained as in Theorem \ref{prorep1} and let $(\pi_\phi^\prime, V_\phi^\prime, K^\prime)$ be another such triplet. Then there exists a unitary $U: K_\phi \rightarrow K^\prime$ such that $V_\phi^\prime = UV_\phi$ and  $U\pi_\phi(a)U^* = \pi_\phi^\prime(a)$, for all $a \in \mathcal{A}.$
	\end{corollary}


    	\begin{center}
		\textbf{Statements and Declarations}
	\end{center}
	
	\textbf{Funding}\\
	The first author's research is supported by the Prime Minister’s Research Fellowship (PMRF ID: 2001291) provided by the Ministry of Human Resource Development (MHRD), Govt. of India.\\

	\textbf{Data Availability}\\
	Data sharing not applicable to this article as no datasets were generated or
	analysed during the current study\\
	
	\textbf{Conflict of interest}\\
	On behalf of all authors, the corresponding author states that there is no conflict of interest.

	\bibliographystyle{amsplain}

\end{document}